\documentclass[12pt]{amsart}

\usepackage{TDefn}

\title{Stability in the homology of Deligne--Mumford compactifications}

\thanks{The author is partially supported by NSF-Grant No. DMS-1903040.}

\date{\today}
\author{Philip Tosteson}
\address{Department of Mathematics, University of Chicago, Chicago, IL}
\email{\href{mailto:ptoste@math.uchicago.edu}{ptoste@math.uchicago.edu}}
\urladdr{\url{https://math.uchicago.edu/~ptoste/}}

\begin{document}

\begin{abstract}
   		Using the theory of $\FSop$ modules, we study the asymptotic behavior of the homology of $\bMgn$,  the Deligne--Mumford compactification of the moduli space of curves, for $n\gg0$.   An $\FSop$ module is a contravariant functor from the category of finite sets and surjections to vector spaces.    Via maps that glue on marked $\bP^1$'s, we give the homology of $\bMgn$ the structure of an $\FSop$ module  and bound its degree of generation. 
   		   As a consequence, we prove that the generating function $\sum_{n}  \dim(H_i(\bMgn)) t^n$ is rational, and its denominator has roots in the set $\{1, 1/2,  \dots,  1/p(g,i)\},$  where $p(g,i)$ is a polynomial of order $O(g^2 i^2)$.   We also obtain restrictions on the decomposition of the homology of $\bMgn$ into irreducible  $\bS_n$ representations.   
\end{abstract}
	\maketitle

	\section{Introduction}
	  		In this paper we study $H_i(\bMgn, \bQ)$, the homology of the Deligne--Mumford moduli space of  stable marked curves,  from the point of view of representation stability.    The space $\bMgn$ is a natural compactification of  the moduli space of smooth curves with $n$ marked  points,  obtained by allowing families of smooth curves  to degenerate to singular curves with double points.   The symmetric group $\bS_n$ acts on $\bMgn$ by relabeling the marked points, so that if we fix $i$ and $g$  we obtain a sequence of symmetric group representations $n \mapsto H_i(\bMgn,  \bQ)$.   
	 
	Our aim is to understand the asymptotic behavior of these $\bS_n$ representations for $n \gg 0$.     
	  		  The following theorem gives applications of our main result.
	  	 \begin{thm}\label{hilbert}
		Let $i,g \in \bN$, and let  $C = 8 g^2 i^2 + 29 g^2 i + 16 g i^2 + 21 g^2 + 10 g i - 6 g.~$
 Then the following hold.
		\begin{enumerate}
			\item   The generating function for the dimension of $H_i(\bMgn, \bQ)$ is rational and takes the form $$\sum_n \dim H_i(\bMgn, \bQ) t^n = \frac{  f(t)}{\prod_{j = 1}^{C} (1 - j t)^{d_j} }$$for some polynomial $f(t)$ and $d_j \in \bN$.  
			In particular, there exist polynomials  $p_1(n), \dots, p_{C}(n)$   such that for $n \gg 0$  we have $$ \dim H_i(\bMgn,\bQ) = \sum_{j=1}^C  p_{j}(n)  j^n.$$
			\item  Let  $\lambda$ be an integer partition of $n$. If the irreducible $\bS_n$ representation  corresponding to $\lambda$ occurs in the decomposition of $H_i(\bMgn, \bQ)$,  then $\lambda$ has length $\leq  C$.  In other words, the Young diagram of $\lambda$ has $\leq C$ rows. 
			\item  Let $\lambda = \lambda_1 \geq \lambda_2 \geq \dots \geq \lambda_C $ be an integer partition of $k$, and $\lambda + n$  be the partition  $\lambda_1 + n \geq \lambda_2 \geq \dots \geq \lambda_C$.  The multiplicity of $\lambda + n$ in $H_i(\widebar \cM_{g, n+k}, \bQ)$,   $$n \mapsto \dim \Hom_{\bS_{n+k}}(M_{\lambda + n},  H_i(\widebar \cM_{g, n + k}, \bQ)),$$ is bounded by a polynomial of degree $C - 1$.  
		\end{enumerate}
	\end{thm}

	  		  \subsection{Main Result}\label{actiondef}
	  		  
	  		To establish Theorem \ref{hilbert} we use techniques from the area known as representation stability.  Namely, we extend the action of the symmetric groups on $H_i(\bMgn, \bQ)$ to the action of a category, and we prove that the homology groups are finitely generated under this action.  Finite generation then constrains the behavior of $H_i(\bMgn,\bQ)$ for $n \gg 0$.

		Let $\FS$ be the category of finite sets and surjections.  The objects of $\FS$  are natural numbers $n \in \bN$.  A map $f: m \to n \in \FS(m,n)$ is a surjection $f: [m] \to [n]$.  Here $[n] := \{1, \dots, n\}$.    An \emph{$\FSop$ module},  or an \emph{action of $\FSop$} on a sequence of vector spaces $V_n$,  is a functor from $\FSop$ to the category of vector spaces, denoted $n \mapsto V_n$.

  We give $\{ H_i(\bar \cM_{g,n})\}_{n \in \bN}$  the structure of an $\FSop$ module.  Concretely, this means that for every surjection $f:  [n] \to [m]$,  we define a map $$f^*:  H_i(\bMgn)  \leftarrow  H_i(\bar \cM_{g,m}),$$  such that $(f \circ g)^* =g^* f^*$ and $\id_{[n]}^* = \id_{H_i(\bar \cM_{g,m})}$.    
		  
		      We describe  $f^*$ in two special cases,  which suffice to determine it in general.   In these cases, $f^*$ is the map on homology induced by a map of spaces, $F^* :  \bMgn  \leftarrow  \bar \cM_{g,m}$.

		 \begin{enumerate} \item Let $f$ be a bijection.  Then  $F^*$  is the map that takes a stable marked curve $C$ and permutes its marked points by precomposing with $f$.  
		  
		  \item Let $f:  [n+1] \to [n]$ is the surjection defined by $f(n+1) = n$ and $f(i) = i$  otherwise.  Given $C \in \bMgn$, define   $F^*(C) $ to be the curve obtained by gluing a copy $\bP^1$ to the $n^{\rm th}$  marked point of $C$.  We mark $F^*(C)$ by keeping the marked points  $p_1, \dots, p_{n-1} \in C$  and marking two new points $p_n, p_{n+1} \in F^*(C) - C$.  Then $F^*(C) \in \overline \cM_{g,n+1}$ and $F^*: \overline\cM_{g,n} \to \overline \cM_{g,n+1}$ is the corresponding map.  
\end{enumerate}
To determine $f^*$ for an arbitrary $f: [n] \to [m]$, factor $f$ as a composition of permutations and surjections of the form $(2)$.  Proposition \ref{actionwelldef}, stated and proved in \S2, shows that this action is well-defined.   Combinatorially, the choice of such a factorization is related to the construction of a binary forest with $n$ leaves and $m$ roots.  Accordingly, in \S 2 we define a category of binary forests, $\BT \op$, which acts on the moduli spaces and induces the $\FSop$ action on homology.

	  		We say that an $\FSop$ module, $n \mapsto V_n$ is \emph{finitely generated in degree $\leq C$} if there is a finite list of classes $\{v_i \in V_{d_i}\}$ with $d_i \leq C$,   such that every $V_n$ is spanned by classes of the form $f^* v_i$.   Our main theorem states that the $\FSop$ module we construct is finitely generated.

	  	\begin{thm}\label{Mgnbar}
			Let $g,i \in \bN$.  Then the $\FS \op$ module  $$n \mapsto H_i(\bMgn, \bQ)$$  is a subquotient of an extension of $\FS \op$ modules that are finitely generated in degree  $\leq 8 g^2 i^2 + 29 g^2 i + 16 g i^2 + 21 g^2 + 10 g i - 6 g$.  
		\end{thm}

		   Theorem  \ref{hilbert} is a consequence of  Theorem \ref{Mgnbar} in combination with  results on  finitely generated $\FSop$ modules due to Sam and Snowden \cite{sam2017grobner}.

	\begin{rmk}[Relation to the Tautological Ring]
		The category $\FSop$ acts on the homology of $\bMgn$ through maps  that glue on copies of $\bP^1$ with two marked points.  These maps are a  small part of the full operadic structure on $H_\bdot(\bMgn)$ generated by all gluing maps.  The \emph{tautological ring}  is the subring of  $H^\bdot(\bMgn)$  generated by the image of all of the fundamental classes $[\bMgn]$ under gluing maps and cup products. 
		  In some sense, Theorem \ref{Mgnbar} says that for $i,g$ fixed, all of the classes in $H_i(\bMgn)$ are tautological \emph{``relative''} to a finite list of classes, using only maps that glue on copies of $\bP^1$ with $2$ marked points.
	\end{rmk}

	\subsection{Stability} 
	
	 		Although the dimensions  $\dim H_i(\bMgn,)$  grow exponentially in $n$, and therefore do not stabilize in a naive sense,  Theorem \ref{Mgnbar} implies that there exists a constant $N$ such that the $\bS_n$ representations $H_{i}(\bMgn,\bQ)$ are completely determined by the vector spaces $\{H_i(\mathcal  \bar\cM_{g,m},\bQ)\}_{m \leq N}$  and the algebraic structure they inherit from surjections $[m] \onto [m']$. 
	 		
	 		 For $r \in \bN$, let $\FS_r \op$ be the full subcategory of $\FSop$  spanned by sets of size $\leq r$.  We may restrict  an $\FS \op$ module  $M$  to an $\FS_r \op$  module, denoted $\Res_r M$.  The functor $\Res_r$ has a left adjoint $\Ind_r$, which takes an $\FS_r \op$ module to the $\FS \op$ module freely generated by it modulo relations in degree $\leq r$. 
	 		
	 		\begin{thm}\label{stability}
	 		 	Let $i, g \in \bN$. There exists $N \in \bN$ such that the natural map of $\FSop$ modules $$\Ind_N  \Res_N  H_i(\bar \cM_{g,-}, \bQ)  \to H_i(\bar \cM_{g,-},\bQ)$$   is an isomorphism.  In particular,  any presentation of the $\FS_N \op$ module $\Res_N  H_i(\bar \cM_{g,-},\bQ)  $  gives a presentation of the $\FSop$ module  $ H_i(\bar \cM_{g,-},\bQ)$. 
	 		\end{thm}
	 		
	 		\begin{rmk} Note that  $\FS_r$ is a finite category and $(\Ind_r M)_n$  can be described as a colimit   $$(\Ind_r M)_n = \colim_{ m \in (n/\FS_r)\op} M_m,   $$ where $n/\FS_r$ denotes the over-category.    Thus, Theorem \ref{stability} says that  $H_i(\bar \cM_{g,n})$  is determined by a finite amount of algebraic data.   \end{rmk}
	 		 
	 		 Theorem \ref{stability} follows from Theorem \ref{Mgnbar} and a Noetherianity result due to Sam and Snowden \cite{sam2017grobner}.

\begin{notation} For the remainder of the paper, all homology and cohomology will be implicitly taken with $\bQ$ coefficients. \end{notation}

\subsection{Relation to other work}

		Our work is motivated by the approach to representation stability  introduced by Church, Ellenberg and Farb, which  uses modules over $\FI$, the category of finite sets and injections  \cite{CEF}.  
		The theory of $\FI$ modules has been used by Jim\'enez Rolland to study the homology of $\Mgn$ \cite{rolland2013cohomology},   and by Maya Duque and Jim\'enez Rolland  to study the real locus of $\bMon$  \cite{rolland2015representation}.    Because the homology of $\bMgn$ grows at an exponential rate, it cannot admit the structure of a finitely generated $\FI$ module, and so a larger category is needed to control the homology of the full compactification.

		Using  an explicit presentation of the cohomology ring $H^\bdot(\bMon)$  given in  \cite{etingof2005cohomology},   Sam defined an action of $\FS \op$ on the cohomology of $\bMon$, and proved that it was finitely generated.  Our work was motivated by his suggestion that there could exist a finitely generated $\FSop$ action on the cohomology of $\bMgn$ for general $g$.
		
		  Sam and Snowden showed that $\FSop$ is Noetherian (submodules of finitely generated modules are finitely generated),  and described the Hilbert series of finitely generated $\FSop$ modules \cite{sam2017grobner}.  We use their results to deduce concrete implications from Theorem \ref{Mgnbar}.\\
		 Proudfoot and Young have also used $\FSop$ modules to study the intersection cohomology of a space closely related to $\bMon$ \cite{proudfoot2017configuration}.  The $\FSop$ module they construct appears similar to our construction in the case $g = 0$.    The statement of our Theorem \ref{hilbert}  parallels their  Theorem 4.3.  
		
		 In order to produce non-tautological classes, Faber and Pandharipande \cite{faber2011tautological} established restrictions on the $\bS_n$ representations that appear in the tautological ring, which resemble the restrictions on  $\bS_n$ representations we obtain in Theorem \ref{hilbert}.  Our restrictions on representations are weaker, but they hold for all cohomology classes.  This bounds the effectiveness of Faber and Pandharipande's method for producing non-tautological classes. 
	
		Kapranov--Manin \cite{kapranov2001modules} observed that $\bigoplus_{i,n} H_i(\bMgn)$ is a right module over the hypercommutative operad.  This algebraic structure extends the action of $\FSop$ on $H_i(\bMgn)$ for fixed $i$.  
	
\subsection{Acknowledgements}
				We thank Steven Sam for sharing his work on the cohomology of $\bMon$. 
				 We also thank Dan Petersen, John Wiltshire-Gordon and Andrew Snowden for helpful conversations.   We also thank the anonymous referee for suggestions that significantly improved the quality of this paper.

\subsection{Heuristic for Theorem \ref{Mgnbar}}\label{heuristic}
		The following is a heuristic argument that illustrates the ideas involved in the proof of Theorem \ref{Mgnbar}. The argument uses notions that we introduce later, and it is mathematically independent from the rest of the text.  Readers may wish to skip this subsection on a first reading.
		
		 We stratify $\bMgn$ by dual graph $G$. The Borel--Moore homology spectral sequence associated to this stratification bounds the homology of $\bMgn$ in terms of the homology of the strata, $H_i^{\rB \rM} (\cM_G)$.   We wish to show that we only need classes from finitely many strata $\cM_G$ in order to generate all of the classes.   We say that a class in $H_i(\overline \cM_{g,n})$ is \emph{pushed forward from lower degree} if it is  a linear combination of classes of the form $f^* c$,  where $f : [n] \to [n-1] \in \FS(n,n-1)$.  

The stratum $\cM_G$  is a quotient of a product of moduli spaces $\prod_{v \in G} \cM_{g(v),n(v)}$.
 By fibering $\Mgn$ over $\cM_{g,1}$  we show that the Borel--Moore homology of $\Mgn$ vanishes for $n > i +3$, thus only strata $\cM_G$  for which $\sum_v {\rm val}(v) - 3  \leq i$ contribute to $H_i(\bMgn)$. 
		   Thus for $G$ ranging over all graphs whose strata contribute classes to $H_i(\bMgn)$,  the number of vertices of $G$ that  have valence $>3$ and genus $> 0$  is bounded by a function of $g$ and $i$.  So as $n \to \infty$  the number of trivalent genus $0$ vertices of $G$ must increase.
		  
		     We say that a stable graph  $H$ \emph{has an external $\rY$} if it has a genus $0$ trivalent vertex $v$ that is adjacent to two external edges.  The action of $\FSop$ on curves corresponds to gluing trivalent vertices on graphs.  Thus if $H$ has an external $\rY$, then every class $c \in H_i(\cM_H)$  is pushed forward from lower degrees. 

Similarly  if $G$ has two adjacent trivalent genus $0$ vertices $v_1, v_2$ such that each $v_i$ has an external edge, the WDVV relation shows that the classes from $\cM_G$ are homologous to classes from $\cM_H$, where $H$ has an external $\rY$. T hus classes from $\cM_G$ are also pushed forward from lower degree.   

Therefore, to prove finite generation, it is enough to show  that when  the number of trivalent genus $0$ vertices of $G$ is large then either  $(1)$  $G$ has an external $\rY$, or $(2)$ $G$ has two adjacent trivalent genus $0$ vertices $v,v'$, each with an external edge.  Each  trivalent vertex with \emph{no} external edges contributes $1/2$ to $-\chi(G)$.  As the number of trivalent vertices increases,  the bound  $-\chi(G) \leq g-1$ implies thatone of these two possibilities must occur.

		 \subsection{Structure of the paper}  In\S \ref{action} we define the category of binary trees, $\BT \op$ and prove that $\FSop$  acts on $H_i(\bMgn)$.   

		  In formalizing the heuristic argument of \S \ref{heuristic}, we encounter the problem that $\FSop$ does not act on the Borel--Moore homology spectral sequence for the stable graph stratification. Since the category of binary trees, $\bB \bT \op$, which does act, is not known to be Noetherian,  we cannot deduce finite generation of $H_i(\overline \cM_{g,n})$  using the usual stable graph stratification. Therefore in \S \ref{stratification},  we define a coarsening of the stable graph stratification for such that $\FSop$ acts on the associated Borel--Moore homology spectral sequence.   

In \S \ref{graphcombinatorics},  we prove two lemmas  that correspond to the combinatorial part of the heuristic argument.   In \S 5, we review the WDVV relation, and the fact that $H_i^{\BM}(\cM_{g,n})$ vanishes for $n > i+3$.   Finally, in \S \ref{proof}, we combine the results from the previous sections to prove Theorem \ref{Mgnbar} and its corollaries.

 In the final section, \S \ref{furtherquestions}, we ask further questions which are motivated by our results.

	
	\section{The Action of $\FSop$}\label{action}

	 Let  $\bMgn$  be the Deligne--Mumford space of stable genus $g$ curves with $n$ distinct marked points.  The space $\bMgn$  parameterizes genus curves $C$  with distinct marked points $p_1, \dots, p_n \in C$, such that all of the singularities of $C$ are double points (also called nodal singularities), each marked point $p_i$ is smooth,  each genus $0$ component of $C$ contains at least $3$ marked or singular points,  and each genus $1$ component contains at least one marked or singular point.

		\begin{rmk} Since our results concern homology with rational coefficients,  we may work with either the homology of the coarse moduli space or the homology of the Deligne--Mumford stack.  For definiteness,  we will work with the moduli space defined over the complex numbers.  However our methods are algebraic, and should also apply to the $l$-adic cohomology of $\overline \cM_{g,n}$ over any algebraically closed field. \end{rmk} 

It will also be convenient for us to use the space $\overline{\cM}_{g,X}$, where $X$ is a finite set and the marked points of $C$ are labelled by the elements of $X$.  Let $[n] := \{1, 2, \dots, n\}$.  Then $\overline \cM_{g,n} = \overline \cM_{g,[n]}$.  
	
Next we introduce the dual graph of a nodal curve $C \in \bMgn$,  an important combinatorial invariant associated to $C$.  
	\begin{defn} Let $C \in \bMgn$. The \emph{dual graph of $C$}, denoted $G_C$, is the following graph: 

\begin{itemize} \item   The vertices of $G_C$ consist of a vertex for each irreducible component of $C$, together with $n$ additional vertices labeled $1, \dots, n$.  The vertices of $G_C$ corresponding to the irreducible components will be called the \emph{internal vertices}, and the $n$ additional vertices will be called the \emph{external vertices}.

\item  For each singular (double) point of $C$, there is an edge that connects the internal vertices corresponding to the components of $C$ that contain the double point.  These components might be the same, in which case the edge is a loop.

\item There is an edge connecting the $k^{\rm th}$ external vertex to the internal vertex corresponding to the irreducible component of $C$ containing the $k^{\rm th}$ marked point.
\end{itemize}
   The \emph{genus}, $g(v)$, of an internal vertex, $v$, is the genus of its irreducible component.   The \emph{valence}, $n(v)$,  of an internal vertex is the number of edges that are adjacent to it.  
	\end{defn}

	 \subsection{Action of binary trees}
		Before passing to homology, the category $\FSop$ does not naturally act on $\bMgn$.  Instead, we construct an action of a category of binary trees.

\begin{defn} We define the \emph{category of binary trees}, $\BT$,  as follows. The objects of $\BT$ are natural numbers.  A morphism $F \in \BT(m,n)$ from $m$ to $n$ is a forest of binary rooted trees,  with the leaves labelled by $[m]$ and the roots labelled by $[n]$.   The composite of two morphisms $F_1 \in \BT(m,n)$ and $F_2 \in \BT(n,l)$ is the forest obtained by gluing the roots of $F_1$ to the leaves of $F_2$  using the labelling of both by $[n]$,  and erasing the resulting bivalent vertices.
		\end{defn}

\begin{ex}
The binary forests \begin{center} \includegraphics[scale = 2]{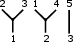} \hspace{.5in}  and  \hspace{.5in} \includegraphics[scale=2]{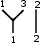}\end{center} define morphisms  from $5$ to $3$ and $3$ to $2$ respectively.  Their composite is the forest  \begin{center}  \includegraphics[scale=2]{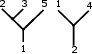}.\end{center}
\end{ex}

For a forest $F \in \BT(n,m)$, the function that takes each leaf to its root is a surjection $h_F: [n] \to [m]$.  The assignment $F \mapsto h_F$ defines a functor $\bB \bT \to \FS$ that realizes $\FS$ as a quotient of $\BT$.

We now define an action of $\bB \bT \op$  on $\bMgn$ by gluing on trees of marked projective lines in the following way.  

\begin{defn}
	Consider $F \in \BT(m,n)$.  We define a variety $L_F$ equipped with marked points $p_1, \dots, p_n$ and $q_1, \dots, q_n$ as follows.  The connected components of $L_F$ correspond to trees in the forest $F$.  For such a tree $T$, the component $L_T$ of $L_F$ is defined as follows:

\begin{itemize}
	\item If $T$ consists of a single root labeled by $r \in [m]$ joined to a single leaf labeled by $s \in [n]$	 joined to a single leaf labeled by $s \in [n]$, then $L_T = *$ and $p_r = q_s = *$.   

\item Otherwise, $L_T$ is the following stable curve.  The irreducible components of $L_T$ are all isomorphic to $\bP^1$ and are in bijection with the vertices of $T$ that are not roots or leaves.  Two of these components meet in a nodal singularity precisely when the corresponding vertices are joined by an edge.   The points $p_i$ and $q_j$ in $L_T$ correspond to the roots and leaves of $T$, respectively, and are chosen to be distinct from one another.  For a root (resp. leaf) labeled by $r \in [m]$ (resp. $s \in [n]$)  the point $p_r$ (resp. $q_s$) is a smooth point in the irreducible component of $T$ corresponding to the vertex of $T$ adjacent to $r$  (resp. $s$).    Since the non-root/leaf vertices have valence $3$, this defines these marked points up to isomorphism.  
\end{itemize} 
\end{defn}

\begin{defn}[Action of $\BT\op$]
		  Given a stable curve $C \in \bMgn$  and a labelled rooted forest  $F \in \bB \bT(m,n)$,  we define the stable curve $F^* C \in \widebar \cM_{g,m}$  to be  $L_F \sqcup_{[m]} C$. In other words $F^*C$ is the curve obtained by gluing  $L_F$ to $C$ along the marked points $\{p_{i}\}_{i \in [m]}$.  We use the marking of $L_F$ by $[n]$ to mark $F^*C$.   Since this construction may be performed in families, it corresponds to a map $F^*:  \widebar \cM_{g,m} \to \bMgn$  for each $F \in \BT(n,m)$, and these maps define an action of $\BT \op$ on $\bMgn$. 

Postcomposing with $H_i(-, \bQ)$,  we obtain a functor  $\BT\op \to {\rm Vect}_\bQ$,  given by $n \mapsto H_i(\Mgn)$. By convention, we take this functor to have the value $0$ in the cases $g = 0, n \leq 2$  and $g =1, n = 0$.  
\end{defn}

 The next proposition defines the $\FSop$ module structure on $H_i(\bMgn)$.
		\begin{prop}\label{actionwelldef}
	  			The functor $n \mapsto H_i(\bMgn)$  factors through the quotient $\BT\op \to \FSop$, hence defines an $\FSop$ module.
  		\end{prop} 
	\begin{proof}  Let $F_1$ and $F_2$  be two forests  inducing the same surjection $h: [n] \to [m]$.   There is a proper family of gluing maps from $\widebar \cM_{g,m}$ to $\bMgn$,  $$ \widebar \cM_{g,m} \times \left( \prod_{i  \in [m], ~ \# h \inv(i) > 1} \widebar \cM_{0, \# h \inv(i) + 1 } \right)  \to \bMgn.$$  The map $F_1^*$  (respectively $F_2^*$) is obtained from this family by evaluating at the point in the second factor defined by the tuple of connected components of  $L_{F_1}$ (respectively $L_{F_2}$). Since the second factor is connected, $F_1^*$ and $F_2^*$ induce the same map on homology. 
\end{proof} 

\begin{remark}
The proof of Proposition \ref{actionwelldef} implies that for any surjection $h: [n] \to [m]$ we have that $h^*: H_i(\overline \cM_{g,m}) \to H_i(\overline \cM_{g,n})$ equals the map induced on homology by the gluing map associated to any tuple of genus $0$ marked curves $(C_i \in \overline \cM_{0, \# h \inv(i) + 1 })$.
\end{remark}

	\section{Stable graph stratification}\label{stratification}

	\begin{defn}
			A \emph{stable graph $G$ of genus $h$ with $n$ external edges} consists of  the following data:

		\begin{itemize}
					\item  A connected graph $G$ and a labeling of a subset of the univalent vertices of $G$ by $1, \dots, n$.  The labeled vertices  are called the \emph{external vertices} and the unlabeled vertices are called \emph{internal vertices}.  The  edges adjacent to external vertices are called \emph{external edges}, and the remaining edges are  called \emph{internal edges}.   We will write $v \in G$ to denote that $v$ is an internal vertex of $G$.  

					\item  A function $g$ from the set of internal vertices to $\bN$, called the \emph{genus function}.  For an internal vertex $v \in G$, we say that $g(v)$ is the \emph{genus of $v$}.
		\end{itemize}
		This data is subject to the following conditions:
\begin{itemize} 
		\item  Each genus $0$ internal vertex is at least trivalent, and each genus $1$ internal vertex is at least univalent. 
		\item  There is an equality  $h^1(G) + \sum_{v \in G} g(v) = h$,  where $h^1(G) = \dim H^1(G)$ denotes the first betti number of $G$.  
\end{itemize}

	 When $G$ is a stable graph of genus $h$ and $n$ marked points, we define $g(G) := h$ and $n(G):= n$.  For each internal vertex $v \in G$, we write $n(v)$  for valence of $v$ and $e(v)$ for the number of self edges of $v$.  
	\end{defn}

\begin{Notation}  Each external vertex is adjacent to a unique external edge.  Because of this correspondence between external vertices and external edges, we may make the following abuse of notation.  When we say that $v$ is a vertex of $G$ without specifying whether it is internal or external,  we always mean that $v$ is an internal vertex of $G$.  Instead of referring to an external vertex, we will typically refer to its corresponding external edge.
\end{Notation}

\begin{ex}
The following stable graph  $G$ has genus $3$ and $1$ external edge.
\begin{center}\includegraphics[scale = 2]{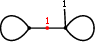} \end{center}
The vertex $v$ marked by $1$ has $g(v) = 1$.  The other two vertices have genus $0$. We do not decorate genus $0$ vertices in our depiction of stable graphs. 
\end{ex}

	Let $G$ be a stable graph, and suppose that  $G'$ is a quotient graph of $G$ such that none of the external edges of $G$ are collapsed in the projection $\pi: G \to G'$.  Then there is a natural stable graph structure on $G'$, defined as follows.  The external vertices  of $G'$  are the images of the external vertices  of $G$ under $\pi$, together with the induced labeling.  The genus of a vertex $w$ of $G'$  is  $\dim H^1(\pi\inv(G)) +  \sum_{v \in \pi\inv (w)} g(v)$.  We say that the stable graph $G'$ is a quotient of $G$.

	   We write $\Stab(g,n)$ for  the set of isomorphism classes of stable graphs of genus $g$ with $n$ external edges.  The set $\Stab(g,n)$ is naturally a poset. We say that  $G \leq G'$ if  $G'$  is a quotient of $G$.  We say that $G \prec G'$ if $G'$ is a quotient of $G$ by a single internal edge.

\begin{defn}
		A \emph{stratification} of a variety $X$ by a poset $P$ is a collection of closed subvarieties $Z(p) \subseteq X$ indexed by $p \in P$,   such that if $p \leq q$  then $Z(p) \subset Z(q)$.  The subsets $Z(p)$ are called \emph{closed strata}.   The \emph{stratum} corresponding to an element $p \in P$  is $S(p):= Z(p) - \bigcup_{q < p}  Z(q)$.      By construction we have that $Z(p)$ equals the set-theoretic union $\bigcup_{q \leq p} S(q)$.  
\end{defn}

 It is well known that the Deligne-Mumford compactification  $\bMgn$ admits a stratification by the poset of stable graphs $\Stab(g,n)$, defined as follows.   Given $C \in \overline \cM_{g,n}$, the dual graph of $C$ is naturally a stable graph, whose genus function assigns to each internal vertex $v$ the genus of the corresponding irreducible component of $C$.     For a stable graph $G \in \Stab(g,n)$, the stratum $\cM_G$ is the locus of stable curves whose dual graph is isomorphic to $G$  (as a stable graph).

We use a coarsening of the standard stratification by $\Stab(g,n)$,  which has the property that $\FSop$ acts on the associated Borel--Moore homology spectral sequence.  This stratification is defined in terms of the operation constructed in the next proposition
	   
	   \begin{prop}\label{closureprop}
	   	Let $G$ be a stable graph. Then there is a unique stable quotient $\widebar G \geq G$  such that no two distinct genus $0$ vertices of $\widebar G$ are connected by an edge, and $\widebar G \leftarrow G$  only identifies edges between genus $0$ vertices.  
	   \end{prop}
	   \begin{proof}
	   		Consider the subgraph consisting of genus $0$ vertices of $G$ and the edges between them.  Choose a minimal spanning tree $T_i$ for each connected component, and define $H$ to be the quotient of $G$ such that each $T_i$ is identified to a point.  
	   		
	   		Any other quotient of $G$ satisfying the hypotheses of the proposition must collapse a minimal spanning forest of the subgraph of genus $0$ vertices, and all quotients by minimal spanning forests are isomorphic.  
	   \end{proof}
	 Notice that $\widebar{ \widebar H} = \widebar H$.  

\begin{wrn} It is not true that if $G \leq H$, then  $\widebar G \leq \widebar H$.  For instance consider the graphs
\begin{center}\includegraphics[scale = 1.75]{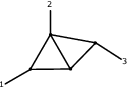}, \hspace{1in} \includegraphics[scale = 1.75]{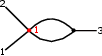}.\end{center}  They are comparable in the partial order,  but  the stable quotients associated to them by Proposition \ref{closureprop},
\begin{center} \includegraphics[scale = 1.75]{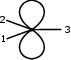}, \hspace{1in} \includegraphics[scale=1.75]{collapsedgraph},\end{center} are incomparable.  
\end{wrn} 
	   
	Let  $\rQ(g,n)$ be the set of stable graphs $G \in \Stab(g,n)$  such that no distinct genus $0$ vertices are connected by an edge.  The next proposition constructs a poset structure on $\rQ(g,n)$ and a stratification of $\bMgn$ by $\rQ(g,n)$.

	   		\begin{prop}\label{posetstrat}
	   			 There is a relation $\leq_Q$ on $\rQ(g,n)$  such that $(\rQ(g,n), \leq_Q)$ is a poset and the map  $H \mapsto \widebar H$ is a surjection of posets $\Stab(g,n) \to \rQ(g,n)$ .   This surjection induces a stratification of $\bMgn$, where the stratum corresponding to $G \in \rQ(g,n)$ is  $S(G) = \bigcup_{H,~ \widebar H = G}  \cM_H.$  
	   		\end{prop}	 
	   		\begin{proof}
	   		  Let $G, H \in \rQ(g,n)$.  We say that  $G \leq_\rQ H$ if there is sequence of stable graphs $G = G_0, G_1, \dots, G_n = H$ such that  either $G_i \prec G_{i+1}$  or $G_i \succ G_{i+1}$ and $G_i$ is obtained from $G_{i+1}$ by collapsing an edge between two genus $0$ vertices. 
	   		  
	   		  By definition, $\leq_\rQ$ is reflexive and transitive.  To prove antisymmetry, suppose $G \leq_\rQ H$ and $H \leq_\rQ G$, and let $G = G_0, G_1, \dots, G_n = H$ be a sequence exhibiting that $G \leq_\rQ H$. 
	   		  
	   		    For a stable graph $J$, we let $s(J) = (s(J)_i)_{i \in \bN} \in \bN^{\oplus \infty}$   be the vector defined as follows.  The $0^{\rm th}$ entry is given by $$s(J)_0 := \sum_{v \in J, g(v) \geq 1} n(v).$$ The $i^{\rm th}$ entry for $i \geq 1$ is the number of vertices of genus $i$.   

	We totally order $\bN^{\oplus \infty}$ reverse-lexicographically, as follows.  If $(n_i), (m_i) \in \bN^{\oplus \infty}$, then $m_i = n_i = 0$ for $i \gg 0$.  So let $k = \max( \{i ~|~ m_i \neq n_i\})$.  We declare $(n_i) \leq (m_i)$ if and only if $n_k \leq m_k$.  

 Now suppose that  $J'$ is obtained from $J$ by collapsing an edge $e$.  If $e$ is an edge between  distinct genus $0$ vertices, then $s(J) = s(J')$.  

Next we claim that if $e$ is any of the following: \begin{enumerate} \item a self edge of a single genus $g$ vertex \item an edge between two distinct genus $\geq 1$ vertices, \item  an edge from a genus $\geq 1$ vertex to a genus $0$ vertex, \end{enumerate}  then $s(J) < s(J')$. In case $(1)$, the $(g+1)^{\rm th}$ entry is the largest entry such that $s(J')$ differs from $s(J)$  and we have $s(J')_{g+1} = s(J)_{g+1} + 1$.  In case $(2)$, suppose the vertices have genus $g_1, g_2$.  The largest entry where $s(J')$ differs from $s(J)$ is the $(g_1 + g_2)^{\rm th}$ entry.  Again the value of this entry increases by $1$.  In case $(3)$ let $n \geq 3$ be the valence of the genus $0$ vertex.  Then $s(J')_0 = s(J)_0 + n - 2$.  

 Thus we see that $s(G_i) \leq s(G_{i+1})$ for all $i$,  with equality if and only if $G_{i}$ and $G_{i+1}$ are related by collapsing an edge between distinct genus zero vertices.  Thus $s(G) \leq s(H)$, and by symmetry $s(H) \leq s(G)$.

	   		   Since $s(H) = s(G)$, for every $i$, the graphs $G_i$ and $G_{i+1}$ are related by collapsing an edge between distinct genus zero vertices.  Otherwise, we would have $s(G_i) < s(G_{i+1})$, contradicting $s(G) =s(H)$.  Thus  $\widebar G_i = \widebar G_{i+1}$ for all $i$, by Proposition \ref{closureprop}.  So $G =  \bar G = \bar H = H$, establishing antisymmetry.
	   		   
	   		   To see that $H \mapsto \widebar H$ is a map of posets, notice that if $H  \leq G$ then $H$ and $G$ are related by a sequence of edge collapses.  By definition the pairs $H$, $\widebar H$ and $G, \widebar G$ are related by a sequence of collapses between distinct genus zero vertices.  Composing these three sequences of edge collapses proves that $\widebar H \leq_Q \widebar G$.
	   		   
	   		   Finally, the surjection $\Stab(g,n) \to \rQ(g,n)$ shows that $$Z(G) := \bigcup_{ J \in \Stab(g,n), ~ \overline J \leq_\rQ G }  \cM_J$$ is a closed subset of $\bMgn$ and  $$S(G) =  Z(G) - \left( \bigcup_{H \in \rQ(g,n),~ H < G}  Z(H)  \right) = \bigcup_{\widebar J = G} \cM_J. $$
	   		\end{proof}
	   		
	   		\begin{rmk}
						The poset structure on $\rQ(g,n)$  constructed in the proof of Proposition \ref{posetstrat} is uniquely characterized as the \emph{minimal} poset structure  on $\rQ(g,n)$ such that the map $\Stab(g,n) \to \rQ(g,n)$ is a map of posets.  
				\end{rmk}

			\begin{defn}   We use   $\widebar \cM_{e ,n,0}$ to denote the moduli space of stable curves $C$ of genus $e$ and $n$ marked points, such that all of the irreducible components of $C$ have genus $0$.  For $G \in \rQ(g,n)$  define  $\widetilde S(G)$ to be the variety
	   			  \[ \widetilde S(G) := \prod_{v \in G, ~ g(v) = 0} \widebar \cM_{e(v),n(v), 0} \times    \prod_{v \in G, ~ g(v) \geq 1} \cM_{g(v), n(v)},\] Recall that $e(v)$ denotes the number of self edges of $v$. 
	   			      
			\end{defn}
			 There is a canonical map  ${\rm glue}: \widetilde S(G) \to S(G)$  given by gluing a tuple of curves $\{C_v\}_{v \in G} \in \widetilde S(G)$ together according to the combinatorics of the graph $G$.  More precisely, $\glue(\{C_v\})$ is the curve obtained from $\bigsqcup_{v \in G} C_v$ by quotienting by the relation that identifies pairs of marked points that correspond to the same edge of $G$.

\begin{defn} Let $G$ be a stable graph.   An \emph{automorphism of $G$} consists of a permutation of the vertices of $G$ and an oriented permutation of the edges of $G$, which preserve the graph structure and  the genus function, and fix the external vertices and edges. (See Remark \ref{automorphisms} below for an alternate definition).  Let $\Aut(G)$ be the group of automorphisms of $G$.

 Let $G \in \rQ(g,n)$. Let $\tilde  G$ be the stable graph obtained from $G$ by collapsing all of the self edges of the genus $0$ vertices of $G$.  Let $T \subset \tilde G$ be the image of the genus $0$ vertices of $G$ under the quotient map $G \to \tilde G$.  Let $A_G \subset  \Aut(\tilde G)$  be the group of automorphisms $f$ of $\tilde G$ such that $f(T) = T$.   
\end{defn}

\begin{ex}
Let $G_1, G_2$ be the graphs 
\begin{center}  
\includegraphics[scale=2]{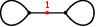}  \hspace{.2in} , \hspace{.2in} \includegraphics[scale=2]{autex2}
\end{center}
respectively.    Then we have that $\Aut(G_1) \iso    \bbZ/2 \ltimes (\bbZ/2 \times \bbZ/2)$ and $\Aut(G_2) \iso \bbZ/2 \times \bbZ/2$, whereas $A_{G_1} \iso \bbZ/2$ and $A_{G_2} = e$.  

Let $G_3$ be the following graph

\begin{center}
 \includegraphics[scale=2]{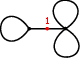}.   
\end{center}
Then $\Aut(G_3) \iso \bbZ/2 \times (\bbZ/2 \ltimes (\bbZ /2 \times \bbZ/2))$  and $A_{G_3} = e$.  
\end{ex}

\begin{rmk}\label{automorphisms} We can describe $\Aut(G)$ more formally as follows.  The data of the graph $G$ can be encoded as a set of half-edges, ${\rm Half}(G)$, a partition $p$ that records which half-edges are adjacent to the same vertex, and an involution $\sigma: {\rm Half}(G) \to {\rm Half}(G)$ that records which half-edges are glued together.  The half-edges fixed by $\sigma$ correspond to the edges that are adjacent to external vertices in our original encoding of $G$. Then $\Aut(G)$ is the subgroup of the permutation group of ${\rm Half}(G)$  consisting of permutations that preserve $p$ and $\sigma$, and which fix pointwise the set of half edges fixed by $\sigma$.  
\end{rmk}

The group $A_G$  acts on $\widetilde S(G)$, by permuting the factors according to the action  of $A_G$ on the vertices of $\tilde G$,  and relabeling the marked points according the action of $A_G$ on the edges of $\tilde G$.

	   		\begin{prop}	 \label{quotient}
	   			For $G \in \rQ(g,n)$, the map  $\widetilde S(G) \to S(G)$  induces an isomorphism  $\widetilde S(G)/A_G \iso S(G)$.  
	   		\end{prop}

	   		\begin{proof}
	   		We first check the map is an isomorphism on $\bC$ points,  then show it is an isomorphism in formal neighborhoods by deformation theory.   Since $S(G)$ is a locally closed subvariety of $\bMgn$, this suffices to show that the map is an isomorphism.  
	   		
	   			 On $\bC$ points  $\widetilde S(G) (\bC)$  is the set of isomorphism classes of collections of marked curves  $\{ C_v\}_{v \in G}$,  such that if $g(v) \geq 1$,  then  $C_v$ is smooth of genus $g(v)$  and if $g(v) = 0$, then $C_{v}$  is has genus $e(v)$ and each irreducible component of $C_{v}$ has genus $0$.   Similarly,  $S(G)(\Spec \bC)$ is the set of isomorphism classes of stable marked curves $C$ whose dual graph $H$ has $\widebar H = G$.  
	   			 
	   			 An element $\tau \in A_G$  acts on $\widetilde S(G) ( \bC)$,  $\tau ( \{C_v\}) = \{ C_{\tau v}\}$, where the marked points are relabeled according to the action of $\tau$ on the edges of $G$.   There are canonical isomorphisms  between $ {\rm glue}(\{C_v\})$ and  ${\rm glue}( \{C_{\tau v}\})$; any curve in $S(G)(\bC)$ is glued from a curve in $\tilde S(G)(\bC)$;  and any isomorphism ${\rm glue}(\{C_v\}) \to {\rm glue}(\{C'_v\})$ factors uniquely as a relabeling of the components of $C_v$  by  $\tau \in A_G$, and isomorphisms between the components $f_v: C_{\tau v} \to C'_{\tau v}$.  Thus  $S(G)(\bC)$ is the quotient of $\widetilde S(G)(\bC)$  by $A_G$.  
	   			 
	   			Next we check that the map is an isomorphism in formal neighborhoods.  Let  $\{C_v\}_{v\in G} \in  \tilde S(G)( \bC)$,  and let $C := \glue(\{C_v\})$  have dual graph $H$.   We show that $$\tilde S(G)/A_G  \to S(G)$$ induces an isomorphism between the formal neighborhood of $\{C_v\}_{v \in G}$ and the  formal neighborhood of $C$.
	   			
	   			Let $\{p_i\}_{i= 1}^n$ be the marked points of $C$.   From the deformation theory of marked stable curves, the formal neighborhood of $C$ is isomorphic to the completion at $0$ of $X$,  where $$X \subset  \Ext^1(\Omega_C(\sum_i p_i),  \cO_C) $$  is the following union of linear subspaces.  There is a subspace for each spanning forest of the subgraph  $H_0 \subset H$ consisting of genus $0$ vertices and edges between them.  The subspace corresponding to the spanning forest $F$  is  $$  \ker \left(   \Ext^1(\Omega_C(\sum_i p_i),  \cO_C) \to  \prod_{e \in \rm {Edges}(H - F)} \Ext^1(\Omega_{\hat C_e}, \cO_{\hat C_e})\right),  $$	where $\widehat C_e$ is the formal neighborhood of the double point $e \in C$ corresponding to $e$.    
	   			
	   			 Similarly,  the formal neighborhood of $\{C_v\}_{v \in G}$  is isomorphic to the completion of $\tilde X$ at $0$ where $$\tilde X \subset \prod_{v \in G} \Ext^1( \Omega_{C_v}(\sum_{u_v } q_{u_v}), \cO_{C_v})$$ is a union of linear subspaces  (the sum  is over edges $u_v$ adjacent to $v$).  Again, $\tilde X$ is a union of subspaces, one for each spanning forest of $H_0$.  The subspace corresponding to $F$ is $$ \ker\left(  \prod_{v \in G  }   \Ext^1(\Omega_{C_v}(\sum_{u_v}  q_{u_v}), \cO_{C_v})   \to  \prod_{e \in {\rm Edges}(H_0 - F)}\Ext^1(\Omega_{\hat C_e}, \cO_{\hat C_e})  \right).$$
	   			 From these identifications,  and the fact that $$\prod_{v \in G} \Ext^1(\Omega_{C_v}(\sum_{u_v} q_{u_v}), \cO_{C_v}) =   \ker\left(\Ext^1(\Omega_C(\sum_i p_i), \cO_C) \to  \prod_{e \in {\rm Edges}(H - H_0)} \Ext^1(\Omega_{\hat C_e}, \cO_{\hat C_e})\right)$$ it follows that the map is an isomorphism on formal neighborhoods.
	   		\end{proof}

Next, we describe how $\BT\op$ acts on the poset $\rQ(g,n)$ and on the topological spaces $\bigsqcup_{G \in \rQ(g,n)} S(G)$  and  $\bigsqcup_{G \in \rQ(g,n)} \tilde S(G)$.  These actions induce an action on Borel--Moore homology spectral sequences.
	  
\begin{defn} 	The category $\BT \op$ acts on the poset $\Stab(g,n)$  by gluing on trees.  More precisely  given $F \in \BT(n,m)$ there is a map $\Stab(g,m) \to \Stab(g,n)$,  denoted $G \mapsto F^*G$, where $F^* G$ is the stable graph obtained from $G$ by gluing $F$ to $G$ using the identification of the external vertices of $G$ with $[m]$, and erasing bivalent vertices.  The action of $\BT\op$ on $\Stab({g,n})$ induces an action of $\BT\op$ on $\rQ(g,n)$ in the sense that there is a unique map $F^*: \rQ(g,n) \to \rQ(g,n)$  such that the diagram
\begin{center}
 \begin{tikzcd}
{{\rm Stab}(g,n)} \arrow[d, two heads] \arrow[r, "F^*"] & {{\rm Stab}(g,m)} \arrow[d, two heads] \\
{{\rm Q}(g,n)} \arrow[r, "F^*"]                         & {{\rm Q}(g,m)}                        
\end{tikzcd}
\end{center}
commutes.   Here the two vertical maps are given by $G \mapsto \overline G$.   
\end{defn}

\begin{ex}
Consider the following forest  $F \in \BT(6,3)$ and  stable graph $G$.
\begin{center}
\includegraphics[scale=2]{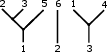} \hspace{1in} \includegraphics[scale=2]{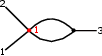}
\end{center}
We may consider $G$ as either an element of $\Stab(2,3)$ or $\rQ(2,3)$.   In these two cases,  $F^* G$ is given respectively by the following two graphs
\begin{center}
\includegraphics[scale=2]{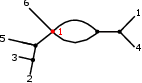}, \hspace{.5in} \hspace{.5in}  \includegraphics[scale=2]{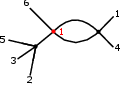}.
\end{center}
 
\end{ex}

\begin{defn}   	The category $\BT\op$ acts on   $\bigsqcup_{G \in \rQ(g,n)} S(G)$ and $\bigsqcup_{G \in \rQ(g,n)}\tilde S(G)$, as follows.    For every $F \in \BT(m,n)$, the gluing map $F^*:  \overline \cM_{g,n} \to  \overline \cM_{g,m}$ induces a map $S(G) \to  S(F^*G)$ and a map $\tilde S(G) \to \tilde S(F^*G)$ for every $G \in \rQ(g,n)$.  Taking the disjoint union of these maps we obtain the desired action.   
\end{defn}

		We let $H_i^{\rm BM}(-)$ denote Borel--Moore homology, and $H^i_c(-)$ denote compactly supported cohomology.  For any finite type variety $X$, we have that $H_i^{\rm BM}(X) = H^i_c(X)^*$.  The variety $\bMgn$ is compact, hence $H_i(\bMgn) = H_i^{\rm BM}(\bMgn)$.

There is an increasing filtration of $\bMgn$ by closed subsets, $Z_i := \bigcup_{G, ~\dim S(G) \leq i}  S(G)$ for  $i \in \bN$.  These closed subsets define a filtration on the Borel--Moore  chains of $\bMgn$, inducing a Borel-Moore homology spectral sequence  
$$E^1_{p,q} =  \bigoplus_{G \in \rQ(g,n), ~\dim S(G) = p}  H_{p+q}^{\BM}(   S(G))  \implies H_{p+q}^{\BM}(\bMgn).$$
This spectral sequence is dual to a spectral sequence for compactly supported cohomology,  which is used more frequently, see for instance Equation 3 of \cite{petersen2017spectral}.  
The next proposition states that  the isomorphism of Proposition \ref{quotient} is compatible with the action of $\FSop$ and the spectral sequence of this stratification.

	   		\begin{prop}\label{compatible}
	   			  The actions of $\BT \op$  on  $\bigsqcup_{G \in \rQ(g,n)} S(G)$ and $\bigsqcup_{G \in \rQ(g,n)}\tilde S(G)$  induce an $\FSop$ module structure on Borel--Moore homology  such that    $$\bigoplus_{G \in \rQ(g,-)}  H_\bdot^{\BM}( \tilde  S(G)) \to \bigoplus_{G \in \rQ(g,-)}  H_\bdot^{\BM}(   S(G))$$ is a surjection of $\FSop$ modules, and the Borel--Moore homology spectral sequence $$\bigoplus_{G \in \rQ(g,-)}  H_\bdot^{\BM}(   S(G))  \implies H_{\bdot}(\overline \cM_{g,-})$$ is a spectral sequence of $\FSop$ modules. 
	   		\end{prop}
	   		\begin{proof}
	   			  By construction, the action of $\BT \op$ preserves the stratification and acts via proper maps. Thus it induces maps of spectral sequences.  We show that on Borel--Moore homology,  the action factors through $\FSop$. Given a forest $F \in \BT(n,m)$ the map $\rQ(g,m) \to \rQ(g,n)$ only depends on $h_F \in \FS(n,m)$.  Given a surjection $h : [n] \to [m]$, there is a family of gluing maps $S(G) \times \prod_{i, h\inv(i) > 1}  \overline \cM_{0, \#h\inv(i) + 1} \to S(h^* G)$.   Each forest $F$ with $h_F = h$ corresponds to an element of $\overline \cM_{0, \#h\inv(i) + 1}$, such that the gluing map $F^*$ is obtained by evaluating at that element.  Since $\bMon$ is connected, every forest induces the same map on Borel-Moore homology.   

The map $H_i^{\BM}(\tilde S(G)) \to H_i^{\BM}(S(G))$ is a surjection, because by Proposition \ref{quotient} we have that $H_i^{\BM}( S(G)) = H_i^{\BM}(\tilde S(G))/A_G$.
	   		\end{proof}

\section{Bounding Graphs}	\label{graphcombinatorics}
	
	The two lemmas in this subsection are combinatorial.  The first says that as $n \to \infty$ any stable graph $G \in \rQ(g,n)$  must either contain a genus $\geq 1$ vertex of high valence,  or a genus $0$ vertex with many external edges.  

	\begin{lem}\label{finitelymanygraphs}
			Let  $f(i,e,a)$ be a linear function  $f(i,e,a) = r i + s e + t a + u$  with  $r,s,t \geq 0$.  
 If $G$ is a stable genus $g$ graph such that:
				\begin{enumerate}
					\item There are no edges between distinct genus $0$ vertices
					\item Every genus $\geq 1$ vertex has  valence  $\leq i +1$
					\item  Every genus $0$ vertex with $\leq e$ self edges and $\leq a$ edges to other vertices, has  $ \leq f(i,e,a)$ external edges.  \end{enumerate}
			Then the total number of external edges of $G$ is bounded by: $$n(G) \leq  (i+1) g +   (i+1) g f(i, g, (i+1)g).$$ 
	\end{lem}
	\begin{proof}
		 For a vertex $v$,  let $e(v)$ be the number of self edges and  $a(v)$ be the number of edges to a distinct vertex.

		 Since every external edge must be adjacent to at least one vertex that is either genus $0$ or genus $\geq 1$,  by the hypotheses $(2)$ and $(3)$ we have $$n(G) \leq \sum_{v \in G, g(v) \geq 1}  (i+1)  + \sum_{v \in G,  ~g(v) = 0}   f(i,e(v),a(v)). $$  
	The first term of the sum is $\leq (i+1)g$, because there are at most $g$ vertices of genus $\geq 1$.

Next we bound the second term.  Note that  for every genus $0$ vertex $v$, we have $e(v) \leq g$,  because the genus of $G$ is  $g$.  We also have $a(v) \leq (i+1) g$,  because any genus $0$ vertex is adjacent to a vertex of genus $\geq 1$.    Thus $f(i,e(v), a(v)) \leq f(i, g, (i+1)g)$.  
		 
		 Now either $G$ has only one genus $0$ vertex,  or every  genus $0$ vertex  must be adjacent to a genus $\geq 1$ vertex.    In the first case $n(G) \leq f(i, g, 0)$,  and in the second we have  $$ \# \{\text{  genus $0$ vertices }\} \leq  (i+1) g $$  So either $n(G) \leq f(i,e,0)$,  or $$n(G) \leq  (i+1) g +  (i+1) g f(i, g, (i+1)g) ,$$ which proves the claim.  
	\end{proof}

		The next lemma concerns stable graphs $H$ containing only genus $0$ vertices.  Roughly, it says that as $n(H) \to \infty$,  either $\dim S(H)$ becomes large, or the number of trivalent vertices of $H$ grows  and consequently $H$ is forced to contain certain subgraphs.

	\begin{lem}\label{forcing}
		Let  $J$ be the stable graph consisting of a single genus $0$ vertex,  $e$ self edges, and $n-e$ labelled external edges.    Assume we are given a partition of  $[n-e]$ into sets $A$ and $B$ of size $a$ and $b$ respectively.   If $b >  13a + 16i + 8e - 7$, then every stable graph $H$ with $\widebar H = J$  satisfies  at least one of the following:
		\begin{enumerate}
			\item[(i)] The sum $\sum_{v \in H}  (n(v) - 3) $ is $> i$ 
			\item[(ii)] There is a trivalent vertex $v$ adjacent to $2$ external edges in $B$
			\item[(iii)] There  are two adjacent trivalent vertices $v, v'$  such that both $v$ and $v'$ are adjacent to external edges in $B$.
		\end{enumerate}
	\end{lem}
	\begin{proof}
	  Let  $H$ be a stable graph with $\widebar H = J$.  For a vertex $v \in H$, let $a(v)$, $b(v)$, and $e(v)$ denote the number of edges adjacent to $v$ that are respectively external in $A$, external in $B$, and self edges.  To prove the contrapostive, we assume that $H$ satsifies the following conditions:
		\begin{enumerate}
			\item  $\sum_{v \in H}  (n(v) - 3) \leq i, $  
			\item every trivalent vertex $v \in H$ has $b(v) \leq 1$
			 \item there are no two adjacent trivalent vertices $v, v' \in H$   with $b(v), b(v') \geq 1$.  
		\end{enumerate} 		
With these assumptions, we show that $b \leq   13a + 16i + 8e - 7$. 

	First, by $(1)$,  $H$ has $\leq i$  vertices of valence $> 3$.

	 Next by retracting external edges and computing the  Euler characteristic of $H$ in two different ways, we have  $$ e - 1 = - \chi(H)=   \sum_{v \in H} ( 1/2 ~\#\{ \text{non external edges of $v$}  \} - 1).$$  The first equality follows because $J$ is obtained from $H$ by contracting a tree.  The second equality is obtained by counting the number of edges minus the number of vertices.  

  Next, we have that the number of non external edges of $v$ is $n(v) - a(v) - b(v)$. Breaking up the sum by $n(v)$ and $a(v)$, and $b(v)$ we obtain:   
	  \[\begin{split} e-1 ~= ~& \sum_{v, ~n(v) > 3}  ( 1/2 ( n(v) - a(v) - b(v) ) - 1) ~~ +   \sum_{v, ~n(v) = 3, ~a(v) \geq  1}   ( 1/2 ( n(v) - a(v) - b(v) ) - 1) \\ +~ &\sum_{v, ~n(v) = 3, ~a(v) = 0, ~b(v) = 1}  0  ~~ +  \sum_{v, ~n(v) = 3, ~a(v) = 0,~ b(v) = 0}  1/2.    \end{split} \] Here we used $(2)$ to restrict the sum to the cases shown.   From the bound on vertices of valence $> 3$, and the fact that $1/2(n(v) - a(v) - b(v) - 1) \geq -1$,  we have:     \[ e-1 \geq -i  - a + 0 ~+~   \sum_{v, ~n(v) = 3,~ a(v) = 0,~ b(v) = 0} 1/2.\] Rearranging this inequality we see \[ \#\{\text{trivalent  vertices $v$,  such that $a(v) = b(v) = 0$}\} \leq 2(i +a + e - 1). \]
	   Since there are at most $a$ trivalent vertices with $a(v) > 0$, there  are at most $\leq 2(i +a + e -1) +a$  trivalent vertices with $b(v) = 0$.   
	    
	    By $(3)$ every trivalent vertex with $b(v) = 1$ satisfies at least one of the following: \begin{itemize} \item[(I):] $v$ has an edge in $A$ \item[(II):]$v$ is adjacent to a trivalent vertex with $b(v) = 0$ \item[(III):] $v$  is adjacent to a  vertex of valence $> 3$ \item[(IV):] $v$ has a self edge and is the only vertex of $H$.  \end{itemize}
We have the following bounds on the number of trivalent vertices satsifying each condition.
\begin{itemize} \item[(I):] $\leq a$ \item[(II):] $\leq  3 \#\{ \text{ genus } 0 \text{ trivalent vertices with } b(v) = 0\}  \leq 3( 2(i +a + e -1) +a)$ \item[(III):] $$\leq \sum_{v, ~ n(v) > 3}  n(v) \leq  i + 3~ \#\{\text{ $> 3$-valent vertices}\} \leq i + 3i $$ \item[(IV):] $\leq 1$, \end{itemize}
where for case (III)  we have rearranged (1) and used the bound on vertices of valence $>3$.  Combining these bounds,  we have 
	    \[ \#\{\text{trivalent vertices with $b(v) = 1$ } \} \leq a+  3(2(i +a + e - 1)+a) + 4i  + 1 = 10 a + 10 i + 6e - 5.\]
	Thus, in total,  the trivalent vertices contribute at most $13 a + 12 i + 8 e - 7$ to $b$.
	
	    Lastly  we have $$\sum_{v,~n(v) > 3}  b(v) \leq  \sum_{v,~n(v) > 3}  n(v)   \leq i + 3i,$$ as we reasoned in case (III) above.  So the vertices of valence $> 3$ contribute at most  $4i$ to $b$.   In total we see that $b \leq   13a + 16i + 8e - 7$, completing the proof.   
	\end{proof}

	\section{Preliminaries on homology}

Recall that  $\overline{\cM}_{e,n,0}$ is the moduli space parameterizing stable curves of genus $e$ with $n$ marked points such that every irreducible component has genus $0$.  Given a stable graph $J \in \Stab(e,n)$ such that $ \overline J \in \rQ(e,n)$  is a wedge of circles,  we write $\cM_J$ for the corresponding stratum of $\overline{\cM}_{e,n,0}$  and $\overline{\cM}_J$ for its closure.

\begin{defn}
	Let $J_1,J_2 \in \Stab(e,n)$ be two stable graphs with no vertices of genus $>1$.  We say that $J_1$ and $J_2$ differ by a \emph{WDVV exchange} if there is a subgraph  $H \subset J_1$, and an identification of $H$ with the graph
\begin{center}\includegraphics[scale = 1.5]{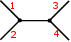},
\end{center}
 such that $J_2$ may be obtained from $J_1$ by replacing $H$ by the subgraph
\begin{center}
 \includegraphics[scale = 1.5]{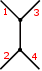}.
 \end{center}
\end{defn}

\begin{ex}The following two graphs differ by a WDVV exchange, applied to the subgraph $H$, indicated by the labelings.  
\begin{center}\includegraphics[scale = 1.5]{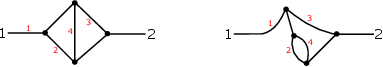}  \end{center}
\end{ex}

The variety $\overline \cM_{0,4}$ is isomorphic to $\bP^1$ via the map that takes $4$ points to their cross ratio. Therefore the three boundary strata of $\overline \cM_{0,4}$ all have the same class in $H_0(\overline \cM_{0,4})$.   This fact is called the WDVV relation.   It is well known that the WDVV relation, combined with gluing morphisms, yields relations between other classes in $\overline \cM_{g,n}$.  The next proposition records the  WDVV relations that we use.  

\begin{prop}[WDVV relations]\label{WDVV}
	Let $J_1,J_2 \in \Stab(e,n)$ be stable graphs with no vertices of genus $> 1$.  Suppose that $J_1$ and $J_2$ differ by a WDVV exchange.  Then the fundamental classes $[\overline \cM_{J_1}], [\overline \cM_{J_2}] \in H_*(\cM_{e,n,0})$ are proportional by a nonzero constant $\lambda \in \bQ^*$.      
\end{prop}
\begin{proof}
Let $J_1- H$ be the graph obtained by removing the $2$ vertices of $H$ and the edge connecting them. Write $J_1 - H$ as a disjoint union of connected stable graphs $\{ K_{i}\}_{i \in I}$ of genus $g_i$ with $n_i$ external edges.  Then there is a gluing map  \[g : \overline{\cM}_{0,4} \times \bigsqcup_{i} \overline \cM_{g_i, n_i} \to \overline\cM_{e,n,0},\]  such that when resticted to the strata of $\overline \cM_{0,4}$ corresponding to the  graphs  
\begin{center}
 \includegraphics[scale = 1.5]{WDVV1} \hspace{.5in}  and \hspace{.5in}  \includegraphics[scale = 1.5]{WDVV2},
 \end{center}
we obtain maps that are finite onto $\overline \cM_{J_1}$ and $\overline \cM_{J_2}$ respectively.   Since  $\overline \cM_{0,4}$ is connected,  the image of the fundamental class of $\bigsqcup_{i} \overline \cM_{g_i, n_i}$ is a nonzero multiple of both $[\overline \cM_{J_1}]$ and $[\overline \cM_{J_2}]$.  
\end{proof}

We recall Deligne's theory of weights  \cite{deligne1974poids,deligne1971theorie},  using the canonical isomorphism $H_j^{\BM}(Y) = H^j_{c}(Y)^* $ to state them in terms of Borel--Moore homology.    For any algebraic variety $Y$, the Borel--Moore homology $H_{j}^{\BM}(Y)$ carries a filtration $$W_{-j} H_j^{\BM}(Y) \subseteq \dots \subseteq W_{0} H_j^{\BM}(Y),$$ which is compatible with with proper pushforwards and the K\"unneth formula.   

\begin{defn}The pure Borel-Moore homology of $Y$ is the subgroup $$\rP H_*^{\BM}(Y):= \bigoplus_{j \geq 0} \rW_{-j} H_{j}^{\BM}(Y) \subseteq H_*^{\BM}(Y).$$  
\end{defn}

Note that the pure Borel--Moore homology equals the pure homology when $Y$ is compact.    The K\"unneth formula for the weight filtration implies that $$\rP H_i(X \times Y) = \bigoplus_{j+k = i} \rP H_j(X) \otimes \rP H_k(Y).$$  

\begin{prop}\label{fundgen}
 	The pure homology of $\overline \cM_{e,n,0}$ is spanned by the fundamental classes of strata $[\overline{\cM}_J]$.   
\end{prop}
\begin{proof}
 Given a curve $C \in \overline \cM_{0,n+2e}$, define $f(C)$ be the curve obtained by quotienting by the relation that identifies the the $n+i$th marked point of $C$ with the $n + e + i$th marked point of $C$ for every $i \in  \{1, \dots, e\}$.  For every $j \in \{1, \dots, n\}$, we define the $j^{\rm th}$ marked point of $f(C)$ to be the image of the $j^{\rm th}$ marked point of  $C$ under the quotient map $C \to f(C)$.   Because $f(C)$ has genus $e$ and every irreducible component has genus $0$, we have defined a point $f(C) \in  \overline \cM_{e,n+e,0}$ .  Since the construction $C \mapsto f(C)$ may be carried out in families of curves over an arbitrary base,  it corresponds to a gluing map $f: \overline \cM_{0,n + 2 e} \to \overline \cM_{e,n,0}$.  

The map $f$ is surjective between projective varieties.  Therefore by Lemma A.4 of \cite{lewis2016survey} the induced map on pure homology is surjective.  

Now, $\rW_{-2i}H_{2i}(\overline\cM_{0,n + 2 e})$, is generated by the fundamental classes of $i$-dimensional stratum closures, $\overline \cM_J$, as can be shown using Hodge theory, see Getzler \cite{getzler1995operads}.  The homology in odd degree vanishes.   For a stable graph $J \in \Stab(0,n + 2e)$,  let ${f}(J) \in \Stab(e, n)$ denote the graph obtained by gluing the external vertices of $J$ as in the definition of ${f}$. Let $\cM_{{f}(J)}$ be the corresponding stratum of $\overline \cM_{e,n,0}$ and let $\overline \cM_{{f}({J})}$ be its closure.   Then $\cM_{f(J)}$ is the image of $\cM_{J}$, so $f([\cM_{J}])$ is a multiple of $[\cM_{f(J)}]$ and so the fundamental classes of strata generate the pure homology of $\overline \cM_{e,n,0}$.  
\end{proof}

Before proving finite generation,  we record the following vanishing statement.
	\begin{prop}\label{noclasses}
	 	Let $g, n \geq 1$.  Then $H_i^{\rB \rM}(\cM_{g,n}) = 0$ for  $i < n -1$.  
	\end{prop}
	\begin{proof}
		We show $H^i_c(\cM_{g,n}) = 0$ for $i < n-1$.  Consider the map  $f: \Mgn \to \cM_{g,1}$ which forgets the last $n-1$ points.  By the Leray spectral sequence for compact support  and base change, it suffices to show that $H^i_c(f\inv (x))$  vanishes for  $i < n-1$.   The fibers are topologically isomorphic to $\Conf_{n-1} (\Sigma_g - p)/B$  where  $\Sigma_g - p$ is the genus $g$ surface with one point removed, and $B$ is a finite group.  By transfer, we have that $H^i_c(\Conf_{n-1}(\Sigma_g - p)/B) \subset  H^i_c(\Conf_{n-1}(\Sigma_g - p))$. Then vanishing of $H^i_c(f \inv(x))$  follows from  Getzler's \cite{getzler1999resolving} spectral sequence converging to the compactly supported cohomology of $\Conf_{n-1} (\Sigma_g - p)$,  or its generalization by Petersen \cite{petersen2017spectral}.  As a graded vector space, the $E_2$ page is $$\bigoplus_{0 \leq k \leq n-2} ( \bQ[-1]  \oplus \bQ^{\oplus {2g}} [-2]   )^{\otimes n -1 - k}   \otimes \bQ^{\oplus c_{n,k}}[-k],$$  where $c_{n,k}$ is a certain unsigned Stirling number.  Since the lowest degree term is in degree $n-1$ , the result follows.
	\end{proof}
 Proposition \ref{noclasses} may also be obtained from the virtual cohomological dimension of $\cM_{g,n}$,  determined by Harer \cite{harer1986virtual}.

	\section{Proof of Finite Generation} \label{proof}

	Throughout this section,  we let  $f(i,e,a) = 13a + 16i + 8e - 7$  and  $p(i,g) = (i+1) g +   (i+1) g f(i, g, (i+1)g)$.   Expanding,  we have $p(i,g) =  8 g^2 i^2 + 29 g^2 i + 16 g i^2 + 21 g^2 + 10 g i - 6 g$.

    We now show that the $\FS \op$ module  $n \mapsto H_i(\bMgn, \bQ)$ is  a subquotient of one that is finitely generated in degrees $\leq p(i,g)$.  In particular, since Sam and Snowden  \cite{sam2017grobner} proved that submodules of finitely generated $\FSop$ modules are finitely generated, we see that the homology is finitely generated.  
	
	\begin{proof}[Proof of Theorem \ref{Mgnbar}]

	 		We use the Borel--Moore homology spectral sequence for the stratification of $\bMgn$ by $\rQ(g,n)$ defined in Section \ref{stratification}.   By Lemma 3.8 of \cite{arapura}, this spectral sequence is compatible with the weight filtration.   By Proposition \ref{compatible},  this is a spectral sequence of $\FSop$ modules and there is a a surjection of $\FSop$ modules  $$\bigoplus_{G \in \rQ(g,n)} H_\bdot^{\rB\rM}\left(  \tilde S(G)  \right) \to  \bigoplus_{G \in \rQ(g,n)} H_\bdot^{\rB\rM}\left(  S(G)  \right).$$ 

    The homology of $\bMgn$ is pure, since it is the quotient of a smooth proper variety by a finite group action \cite{boggi2000galois} (it is the coarse moduli space underlying a smooth and proper Deligne--Mumford stack).  Thus it suffices to show that the pure Borel--Moore homology of $\sqcup_{G} \tilde S(G)$ is generated in degree $\leq p(i,g)$.  

	Let $G \in \rQ(g,n)$ for $n > p(i,g)$, and let $c \in \rP H_i(\tilde S(G))$.  It suffices to show that the class $c$ is \emph{pushed forward from lower degrees}, in the sense that $c$ is a linear combination of classes of the form $f^*(d)$ for some surjection $f: [n] \to [n-1]$.

	 By Lemma \ref{finitelymanygraphs},  since $n(G) > p(i,g)$ and $G$ has no edges between distinct genus $0$ vertices,  one of the following must hold:\begin{enumerate} \item  $G$ has no genus $\geq 1$ vertices of valence $> i+1$, \item $G$ has no genus $0$ vertices with $e$ self edges,  $a$ edges to adjacent vertices and  $> f(i,e,a)$ external edges. \end{enumerate} 
	
Assume $(1)$ holds. Then by the K\"unneth formula we have that  \[  \rP H_i^{\rB\rM}(\widetilde S(G))= \bigoplus_{j:  G \to \bN, ~ \sum_v j(v) = i} \left( \bigotimes_{v \in G, ~g(v) = 0}    \rP H_{j(v)}^{\rB\rM}(\overline  \cM_{n(v), e(v), 0} )\otimes \bigotimes_{v \in G,~ g(v) \geq 1}  \rP H_{j(v)}^{\rB\rM}(\cM_{n(v), g(v)}) \right).\]    Since $G$ has a genus $\geq 1$ vertex with valence $> i+1$,  then by Proposition \ref{noclasses}  we have that $  H_{j}(\cM_{g(v), n(v)}) = 0$ for all $j \leq i$.  Thus $ \rP H_i^{\rB\rM}(\widetilde S(G)) = 0$  and so $c = 0$. Hence $c$ is pushed forward from lower degrees.  
	
		Now assume $(2)$ holds.  Let $w$ be a genus $0$ vertex of $G$ with $e$ self edges, $a$ edges to distinct vertices,  and $b > f(i,e,a)$  external edges.  We define $$X := \prod_{v \neq w, ~ g(v) = 0} \widebar \cM_{e(v),n(v), 0} \times    \prod_{v , ~ g(v) \geq 1} \overline \cM_{g(v), n(v)}.$$ By taking linear combinations, we may reduce to the case where $c$ is of the form  $c =c_w  \otimes c'$,  for $c_w \in \rP H_{j(w)}( \overline \cM_{e, a+b, 0})$ and $j(w) \leq i$  and  $ c' \in   \rP H_{i - j(w)}\left( X \right).$
	
 Let $J$ be the subgraph consisting of $w$ and edges adjacent to $w$.  By Proposition \ref{fundgen},  $\rP H_{j(w)}( \overline \cM_{e, a+b, 0})$  is spanned by fundamental classes of strata $[\overline{ \cM}_H]$ for stable graphs $H \in \Stab(g,n)$  with $\widebar H = J$ and $\dim \cM_H =  j(w) \leq i$.  
It suffices to show that $[\widebar{\cM}_H] \otimes c' = f^*(d)$ for some surjection $f:  [n] \to [n-1]$.     By Lemma \ref{forcing}, applied to the partition of the edges of $J$ into internal and external edges, at least  one of the following  holds for $H$:
		\begin{enumerate}
			\item[(I)] $\sum_{v \in H }  (n(v) - 3) $ is $> i$.  Since this sum is $\dim \cM_H$, this contradicts the assumption that $\dim \cM_{H} \leq i$.  \\
			\item[(II)] There  is a trivalent vertex $v \in H$ adjacent to two external edges, with labels $s,t \in [n]$.   Let $f$ be a surjection $[n] \to  [n-1]$  such that $f(s) = f(t)$. Note that $\BT(n,n-1) \to \FS(n,n-1)$ is a bijection, and let $F$ be the forest corresponding to $f$.  There is a unique $K \in \Stab(e,n-1)$ such that $F^*K = H$. The graph  $K$ is obtained by removing the edges $s,t$ from $H$  and relabelling according to $f$.   Similarly, there is a unique $L \in \rQ(g,n-1)$ such that $f^* L = G$. The graph $L$ obtained by removing one of the external edges adjacent to $w$, and relabeling according to $f$.   

Since $L$ only differs from $G$ at the vertex $w$, there is an identification $\tilde S(L) =\overline \cM_{e(w), n(w)-1,0} \times X.$   Thus we may consider  $[\overline \cM_K] \otimes c'$  to be a class on $\tilde S(L)$.  
Restricting to the subvariety $\overline \cM_K \times X$,  the gluing map $F^*: \tilde S(L) \to \tilde S( G)$ induces an isomorphism  $$\overline \cM_K \times X \to  \overline \cM_H \times X.$$ Thus we have that    $f^* ([\overline \cM_K] \otimes c') = [\overline \cM_H] \otimes c'$.  \\
			
			\item[(III)] There are two adjacent trivalent vertices $v_1, v_2$  in $H$, such that $v_1, v_2$ both have external edges. Identify the subgraph spanned by edges adjacent to $v_1, v_2$ with  \begin{center}\includegraphics[scale = 1.5]{WDVV1}.
\end{center}
Then, applying a WDVV exchange to $H$ using this identification, we obtain a graph $H'$  which has a trivalent vertex with two external edges.    Proposition \ref{WDVV} implies that  $[\widebar{ \cM}_H]  =  \lambda [\widebar{ \cM}_{H'}].$
  Then by case (II) applied to $H'$, we have that $[\cM_H]$  is pushed forward from lower degrees.  \\
		\end{enumerate}
	Therefore we have $[\overline \cM_H] \otimes c' = f^*(d)$ for some surjection $f:  [n] \to [n-1]$, completing the proof.
	\end{proof}

\begin{proof}[Proof of Theorem \ref{hilbert}]
		  The result follows from Theorem \ref{Mgnbar} and the extension of the results of Sam and Snowden recorded   Theorem 4.3 of \cite{proudfoot2017configuration}.
\end{proof}

\begin{proof}[Proof of Theorem \ref{stability}]
	 By Corollary 8.1.3 of Sam and Snowden \cite{sam2017grobner},  $\FSop$ is Noetherian.  Hence, by Theorem \ref{Mgnbar}, the $\FSop$ module $H_i(\cM_{g, -})$ is finitely presented, with presentation $$\bigoplus_{i \in I_1} \bbQ \FS(-,d_i) \to \bigoplus_{j \in I_0} \bbQ \FS(-,d_j),$$  where $\bbQ \FS(-,d)$ is the representable $\FSop$ module in degree $d$, defined by $n \mapsto \bbQ\FS(n,d)$,  the freeb vector space with basis the set of surjections $[n] \to [d]$.   Let $N = \max_{i \in I_1} (d_i)$.  Then we have that \[\Ind_N \Res_N H_i(\cM_{g,-}) = \coker\left(\bigoplus_{i \in I_1}  \Res_N\Ind_N \bbQ\FS(-,d_i) \to \bigoplus_{j \in I_0}  \Res_N \Ind_N \bbQ  \FS(-,d_j)\right). \]  Since the counit $\Ind_N \Res_N \bbQ \FS(-,d) \to \bbQ \FS(-,d)$ is an isomorphism for $d < N$, the result follows.  
\end{proof}

	\section{Further Questions}\label{furtherquestions}
	\subsection{Extensions of Theorem \ref{Mgnbar}}

			
		From Theorem 
		 \ref{hilbert}, we see that there exists $M(g,i) \in \bN$  such that the function $n \mapsto \dim H_i(\bMgn)$ agrees with a sum of polynomials times exponentials for $n \geq M(g,i)$.  What is an effective bound on $M(g,i)$?  Equivalently, what is a bound on the degree of the numerator of the generating function $\sum_n \dim H_i(\cMgn) t^n$  in terms of $g$ and $i$?   
		
		We show that $H_i(\bMgn)$ is a subquotient of a module generated in degree $O(g^2 i^2)$.  We expect that this bound can be a linear bound in $g$ and $i$.  What is the optimal bound?
		
		What is an effective bound on the constant $N$ of Theorem \ref{stability}?
		
	\subsection{Representation theory of $\FSop$}	 Finitely generated $\FSop$ modules are not yet as well understood as finitely generated $\FI$ modules.  By Theorem \ref{Mgnbar}, any new  results on  finitely generated $\FSop$ modules will apply to $H_i(\bMgn)$. 
	
	 		Let $M_\lambda$ be an irreducible representation of $\bS_n$.  What is the decomposition of the projective $\FSop$ representation $M_\lambda \otimes_{\bS_n} \bbQ \FS(-,n)$ in degree $m$?  Equivalently, what is  the degree $m$ term in the plethysm of symmetric functions   $s_\lambda[\sum_{k \geq 1}  h_k]$? 
	
			The dimension functions of $\FI$ modules are eventually polynomial, and there has been much work towards constructing and studying invariants of $\FI$ modules that govern when this occurs.  (For recent applications see \cite{church2018linear}).  Which invariants of $\FSop$ modules govern when the dimension function of a finitely generated $\FSop$ module agrees with a sum of polynomials times exponential?   Which invariants control its presentation degree?

	\subsection{Other operadic actions on $H_\bdot(\bMgn)$}
		Operads give a conceptual interpretation of the $\FSop$ action on $H_i(\bMgn)$.  The homology of moduli space $\bigoplus_{g,n} H_\bdot(\bMgn)$ forms a (modular) operad, and $\bigoplus_{g,n} H_i(\bMgn)$  is a right module over the suboperad generated by $[\bar \cM_{0,3}]$.    This operad is isomorphic to the commutative operad.  Thus we obtain an right action of the commutative operad.  This action corresponds to a representation of its associated $PROP$, which is the category $\FSop$.   Using this approach, we can define other right actions of wiring categories on the homology of moduli space.  For example,  for any $j$, the operad generated by $[\bMgn]$   acts on $\bigoplus_{j,g,n,~  3g - 3 + n - j  = i  } H_j(\bMgn)$.  Is this right module finitely generated?  
		
		Similarly, as first observed by Kapranov--Manin \cite{kapranov2001modules}, for  $g$ fixed, the vector space \newline $\bigoplus_{i,n} H_i(\bMgn)$  is a right module over the \emph{hypercommutative operad}, the operad spanned by the fundamental classes of $\bMon$.  Is this right module finitely generated?
		
		For the answers to these questions to have concrete implications, we need to understand the structure of right modules over these operads (equivalently representation theory of their opposite wiring categories).  If a sequence of graded vector spaces is a finitely generated module over the hypercommutative operad,   is its Hilbert series constrained to have a particular form?

		Is the category $\bB\bT \op$, the opposite of the wiring category of the free operad on a binary commutative operation, Noetherian?  If so, it would be possible to simplify the proof of Theorem \ref{Mgnbar}.  In \cite{barter2015noetherianity},  Barter established the Noetherianity of a category of trees,  but we do not know of a relationship between this category and $\bB \bT \op$.

			\printbibliography

\end{document}